\documentclass[a4paper]{article}

\usepackage{amsmath}
\usepackage{amsfonts}
\usepackage{amssymb}
\usepackage[mathscr]{eucal}
\newtheorem{theorem}{Theorem}

\newtheorem{definition}[theorem]{Definition}
\newtheorem{lemma}[theorem]{Lemma}
\newtheorem{remark}[theorem]{Remark}

\newtheorem{proposition}[theorem]{Proposition}
\newenvironment{proof}[1][Proof]{\noindent\textit{#1.} }{\hfill $\Box$}

\newcommand{\OBSI}{\begin{remark}\begin{rm}}
\newcommand{\OBSF}{\end{rm}\end{remark}}
\newcommand{\DEFI}{\begin{definition}\begin{rm}}
\newcommand{\DEFF}{\end{rm}\end{definition}}

\newcommand{\be}{\begin{eqnarray}}
\newcommand{\en}{\end{eqnarray}}
\newcommand{\bee}{\begin{eqnarray*}}
\newcommand{\ene}{\end{eqnarray*}}

\newcommand{\CC}{\mathrm{C}}
\newcommand{\var}{\varepsilon}

\DeclareMathOperator*{\limK}{lim{\text{-}}\mathrm{K}}

\DeclareMathOperator*{\diam}{diam}
\DeclareMathOperator*{\essinf}{ess.inf}

\DeclareMathOperator*{\supp}{supp}

\begin{document}

\title{Generic zero-Hausdorff and one-packing  spectral measures}

\author{Silas L. Carvalho\thanks{Email: \texttt{silas@mat.ufmg.br}} \ \\\small Departamento de Matem\'atica, UFMG, Belo Horizonte, MG,  30161-970 Brazil \\ \\ and \\ \\ C\'esar  R. de Oliveira\thanks{Email: \texttt{oliveira@dm.ufscar.br}}\\ \small Departamento de Matem\'atica, UFSCar, S\~{a}o Carlos, SP,  13560--970 Brazil}

\maketitle





\begin{abstract} For some metric spaces of self-adjoint operators, it is shown that the set of operators whose spectral measures have simultaneously zero upper-Hausdorff and one lower-packing dimensions contains a dense~$G_\delta$ subset. Applications include sets of limit-periodic operators.  
\end{abstract}
{Key words and phrases}.  {\em Self-adjoint operators, spectral measures,  upper-Hausdorff dimension, lower-packing dimension.}

\maketitle

\section{Introduction} Let $(X,d)$ be a complete metric space of self-adjoint operators acting in a separable Hilbert sapce~$\mathcal{H}$, such that convergence in the metric~$d$ implies strong resolvent convergence.  In three previous papers~\cite{SCC,SCD,SCP}, the present authors have discussed several  generic sets of families of self-adjoint operators, in some instances of the space $(X,d)$, in terms of not only spectral properties, but also of dynamical ones. In such works we have gotten, through different grounds, generic sets of operators with one-dimensional packing spectral measures, but an argument for Hausdorff dimensional properties was missing; it is one of the goals of this work to fill up this gap by presenting a result in terms of what we call {\it fractal dimensions} of the spectrum: we give contributions related to the upper-Hausdorff and lower-packing dimensions of spectral measures. In~\cite{dRJLS} one may find a way to the theory of singularly continuous spectra, dimensions and some of its perturbations.



Although it is already known that, for some families of self-adjoint operators, a typical (in Baire's sense) spectral measure has upper-packing dimension equal to one (see Theorem~1.1 in~\cite{SCP}), we improve such result, in the sense that now the same result is valid for the lower-packing dimension; however, as mentioned before, there was no generic result  about the (upper or lower) Hausdorff dimension. The novel technical argument, encapsulated in Theorem~\ref{P2}, gives information about upper-Hausdorff dimensional properties of spectral measures; since it is immediate to adapt such ideas to obtain the counterpart lower-packing properties, we just present the details of the first case.  It is also important to underline that {\it every} application of the so-called Wonderland Theorem discussed in~\cite{SimonWT}, presenting dense sets of operators with pure point spectrum or absolutely continuous spectrum, can now be converted into a result about the existence of a generic set of operators whose spectral measures are zero upper-Hausdorff and one lower-packing dimensional. The upper-Hausdorff dimension of a Borel measure~$\mu$ will be denoted by $\dim_{\mathrm{H}}^+(\mu)$, whereas its lower-packing  dimension by $\dim_{\mathrm{P}}^-(\mu)$ (such concepts are recalled in Definition~\ref{defDimMu}).

\subsection{Main results}

Next, we present our main result. It should be compared with Theorem~2.1 in~\cite{SimonWT}.

\

\begin{theorem} \label{thmFDW} Let~$0\neq\psi\in\mathcal{H}$, let 
  $\emptyset\neq F\subset\mathbb{R}$ be a closed set and suppose that each of the sets 
\begin{itemize}
\item $C^{\psi;F}_{0{\mathrm{Hd}}}=\{T\in X\mid\dim_{\mathrm{H}}^+(\mu_{\psi;F}^{T})=0\}$, 
\item $C^{\psi;F}_{1{\mathrm{Pd}}}=\{T\in X\mid\dim_{\mathrm{P}}^-(\mu_{\psi;F}^{T})=1\}$,
\end{itemize}
\noindent is dense in~$X$. 
Then, the set $\{T\in X\mid\dim_{\mathrm{H}}^+(\mu_{\psi;F}^T)=0\;\;\mathrm{and}\;\dim_{\mathrm{P}}^-(\mu_{\psi;F}^T)=1\}$ is generic in~$X$.
\end{theorem}

As an illustration, we consider an application to a class of  bounded discrete Schr\"odinger operators acting on~$l^2(\mathbb{Z})$. For a fixed~$r>0$, let~$X^r$ be the set of operators~$T$ with action 
\begin{equation}\label{OSch1}
(T\psi)_n=\psi_{n+1}+\psi_{n-1}+V_n\psi_n\,,
\end{equation}
where the potential $v=(V_n)$ is an arbitrary real bilateral sequence with $|V_n|\le r$ for every~$n\in\mathbb{Z}$. Let $\sigma(T)$ and $\mu^T_\psi$ denote the spectrum of~$T$ and its spectral measure (associated with the vector~$0\neq\psi\in l^2(\mathbb{Z})$), respectively. 
By combining Theorem~\ref{thmFDW} with a specific construction presented in the proof of Theorem~4.1 in~\cite{SimonWT}, we obtain the following result.

\begin{theorem} \label{theorMainSO}
Fix $r>0$.  The set $\{T\in X^r\mid\sigma(T)=[-2-r,2+r]$, $\dim_{\mathrm{H}}^+(\mu_\psi^T)=0$, $\dim_{\mathrm{P}}^-(\mu_\psi^T)=1$, for all~$0\neq\psi\in\mathcal{H}\}$ is generic in~$X^r$.
\end{theorem}       

\OBSI
\label{RIMP}
A well-known fact about discrete Schr\"odinger operators in~$l^2(\mathbb{Z})$, with action~\eqref{OSch1}, is the existence of a common set of cyclic vectors $\{\delta_{0},\delta_1\}$. Now, if for $\zeta\in\{\delta_{0},\delta_1\}$ the spectral measure $\mu^T_{\zeta}$ is zero upper-Hausdorff dimensional, then $\mu^T_{\psi}$ is zero upper-Hausdorff dimensional for every vector $\psi\ne0$ (namely, since $\mu^T_{\zeta}$ is supported on a set of zero Hausdorff dimension and since, for every $\psi\ne0$,  $\mu^T_{\psi}$ is absolutely continuous with respect to $\mu^T_\zeta$, then $\mu^T_{\psi}$ is also supported on a set of zero Hausdorff dimension), which implies that~$\{T\in X\mid\sigma(T)$ is purely zero upper-Hausdorff dimensional$\}$ is a~$G_\delta$ set (the same conclusion is valid for~$\{T\in X\mid\sigma(T)$ is purely one lower-packing dimensional$\}$). Thus, the results stated in Theorem~\ref{theorMainSO} are obtained after showing that the set $\{T\in X^r\mid\sigma(T)=[-2-r,2+r]$, $\dim_{\mathrm{H}}^+(\mu_\psi^T)=0$ and $\dim_{\mathrm{P}}^-(\mu_\psi^T)=1$, for each fixed $0\neq\psi\in\mathcal{H}\}$, is generic in~$X^r$. This is actually what one gets combining Theorem~\ref{thmFDW} with the aforementioned result in~\cite{SimonWT}.
\OBSF

We also apply our results to a class of limit-periodic operators; these are discrete one-dimensional ergodic Schr\"odinger operators, denoted by $H^\kappa_{g,\tau}$, acting in $l^2(\mathbb{Z})$, whose action is given by~\eqref{OSch1}, with 
\begin{equation}\label{VF}
V_n(\kappa)=g(\tau^n(\kappa))\,;
\end{equation} \noindent here, $\kappa$ belongs to a Cantor group~$\Omega$, $\tau:\Omega\rightarrow\Omega$ is a minimal translation on~$\Omega$ and $g:\Omega\rightarrow\mathbb{R}$ is a continuous sampling function, i.e., $g\in\CC(\Omega,\mathbb R)$, the latter endowed  with the norm of uniform convergence. For more details, see~\cite{Avila}.

For each~$\kappa\in\Omega$, let $X_{\kappa}$ be the set of limit-periodic operators $H_{g,\tau}^\kappa$ given by \eqref{OSch1} and \eqref{VF}, endowed with the metric 
\begin{equation}\label{metric3}
d(H_{g,\tau}^\kappa,H_{g',\tau}^\kappa)=\|g-g'\|_\infty\;.
\end{equation} We shall prove the following result.

\begin{theorem}\label{LPO}
  For each $\kappa\in\Omega$, the set~$\{T\in X_\kappa\mid\sigma(T)$ is purely zero upper-Hausdorff and one lower-packing dimensional$\}$ 
  is generic in~${X}_{\kappa}$.
\end{theorem}

\subsection{Countable families of pairwise commuting self-adjoint operators}

We remark that is possible to extend the result stated in Theorem~\ref{thmFDW} for countable families of pairwise commuting self-adjoint operators $T=(T_1,\ldots,T_N)$ acting in a separable Hilbert space~$\mathcal{H}$. The joint resolution of identity is given by~$E(\cdot):=\prod_{j=1}^NE_j(\cdot)$ over the rectangles of the Borel sets~$\mathcal{B}(\mathbb{R}^{N})$; here,  $N$ stands for a natural number or (countable) infinite, and  $E_j(\cdot)$ is the resolution of identity of~$T_j$. For each fixed~$\psi\in\mathcal{H}$ with~$\Vert\psi\Vert=1$, the support of the spectral measure~$\mu_\psi^T(\cdot):=\langle\psi,E(\cdot)\psi\rangle$, denoted by~$\supp(\mu_\psi^T)$, is the intersection of all closed subsets of~$\mathbb{R}^N$ with full~$\mu_\psi^T$ measure ($\mathbb{R}^{\mathbb N}$ with the product topology). We also set $J_N=\{1,2,\, \cdots,N\}$ if~$N\in\mathbb N$, and $J_N=\mathbb N$ in case~$N=\infty$.

\DEFI\label{DIM}
Let~$K$ denote either~$H$ or~$P$, for Hausdorff or packing, respectively. Let~$\mu$ be a probability product-measure  on the Borel sets $(\mathbb{R}^{N};\mathcal{B}(\mathbb{R}^{N}))$ given by $\mu(\cdot)=\prod_{n=1}^N\mu_{n}(\cdot)$. Let~$I=\prod_{n=1}^NI_n\in\mathcal{B}(\mathbb{R}^{{N}})$ be a measurable rectangle. One says that $\dim_{\mathrm{K}}^{\pm}(\mu)$ is minimal if, for each~$n\in J_N$, $\dim_{\mathrm{K}}^{\pm}(\mu_n)=0$. Accordingly, one says that~$\dim_{\mathrm{K}}^{\pm}(\mu)$ is maximal if, for each~$n\in J_N$,~$\dim_{\mathrm{K}}^{\pm}(\mu^n)=n$, where 
\[
\mu^n:=\prod_{k=1}^n\mu_k.
\] 
\DEFF

Denote by~$X$ the collection of such families of countable sequences of pairwise commuting self-adjoint operators, and 
let $d$ be any metric in~$X$ whose convergence implies, for each~$k\in J_N$, strong resolvent convergence; one could set, for instance, 
\[d(T,T^\prime):=\sup_{k\in J_N}D(T_k,T_k^\prime),\]  where 
\[
D(T_k,T_k^\prime):=\sum_{l\ge 1}\min(2^{-l},\Vert (T_k-T_k^\prime)\xi_l\Vert)
\] ($(\xi)_{l\ge 1}$ is an orthonormal basis of~$\mathcal{H}$). Naturally, $(X,d)$ is a complete metric space. The following result is the natural extension of Theorem~\ref{thmFDW} to this setting. 

\begin{theorem} \label{thmFDW1} Let~$\psi\in\mathcal{H}$, with~$\Vert\psi\Vert=1$, let for each $j\in J_N$, $\emptyset\neq F_j$ be a closed set and put $F:=\prod_{j=1}^{J_N}F_j$. Suppose that each of the sets 
\begin{itemize}
\item $C^{\psi;F}_{{\mathrm{min}}}=\{T\in X\mid\dim_{\mathrm{H}}^+(\mu_{\psi;F}^{T})$ is minimal$\}$, 
\item $C^{\psi;F}_{{\mathrm{max}}}=\{T\in X\mid\dim_{\mathrm{P}}^-(\mu_{\psi;F}^{T})$ is maximal$\}$,
\end{itemize}
\noindent is dense in~$X$. 
Then, the set $\{T\in X\mid\dim_{\mathrm{H}}^+(\mu_{\psi;F}^T)$ is minimal and $\dim_{\mathrm{P}}^-(\mu_{\psi;F}^T)$ is maximal$\}$ is generic in~$X$.
\end{theorem} 

One can prove Theorem~\ref{thmFDW1} using adapted versions of the results stated in Section~\ref{sectHW} for functions defined in $\mathbb{R}^n$, with $n\in J_N$.  

The result stated in Theorem~\ref{thmFDW1} is particularly true for the set of normal operators acting in $\mathcal{H}$, which we denote by $Y$; recall that a normal operator~$A$ can be written in terms of a pair~$T_1,T_2$ of commuting self-adjoint operators: $A=f(T_1,T_2)$, where~$f:\mathbb{R}^{2}\rightarrow\mathbb{C}$,~$f(x_1,x_2)=x_1+ix_2$. This also leads to a version of Simon's Wonderland Theorem~\cite{SimonWT} to normal operators.

\begin{theorem} Let $(Y,d)$ be as above, and suppose that each of the sets 
\begin{itemize}
\item $\{A\in Y\mid A$ has purely absolutely continuous spectrum$\}$,
\item $\{A\in Y\mid A$ has pure point spectrum$\}$
\end{itemize}
\noindent is dense in~$Y$. 
Then, the set $\{A\in Y\mid A$ has purely singular continuous spectrum$\}$ is generic in~$Y$.
\end{theorem}

In what follows, we use the remark above in order to extend the result stated in Theorem~3.1 in~\cite{SimonWT} to normal operators. Let $a:=(a_1,a_2)$ be such that~$a_1,a_2>0$, and set $Y^a=\{A\in Y\mid\Vert T_1\Vert\le a_1$, $\Vert T_2\Vert\le a_2\}$. 

\begin{theorem}\label{normal}
Let~$\psi\in\mathcal{H}$ with~$\Vert\psi\Vert=1$ and set $R:=[-a_1,a_1]\times[-a_2,a_2]$. Then, the set~$\{A\in Y^b\mid \supp(\mu_\psi^A)=R$,~$\dim_{\mathrm{H}}^+(\mu_{\psi}^A)$ is minimal,~$\dim_{\mathrm{P}}^-(\mu_{\psi}^A)$ is maximal$\}$ is generic in~$Y^a$.
\end{theorem}

 \subsection{Organization}
 In Section~\ref{HPM}  we recall important decompositions of Borel measures on~$\mathbb{R}$ with respect to Hausdorff and packing dimensions, along with the corresponding spectral decompositions of self-adjoint operators. 
 Section~\ref{sectHW} is dedicated to the construction of suitable $G_\delta$ sets. In Section~\ref{sectHW1} we present the proofs of Theorems~\ref{thmFDW} and~\ref{LPO}.

 Now some words about notation. $\mathcal H$ will always denote a complex separable Hilbert space. 
 $\mathcal{B}(\mathbb{R})$  denotes the collection of Borel sets in~$\mathbb{R}$; $\mu$ will always indicate a finite nonnegative Borel measure on~$\mathbb{R}$, and its restriction to the Borel set~$A$ will be indicated by $\mu_{;A}(\cdot):=\mu(A\cap\cdot)$. The adjective \textit{absolutely continuous} without specification means that~$\mu$ is absolutely continuous with respect to Lebesgue measure on~$\mathbb{R}$.  A nonnegative Borel measure~$\nu$ on~$\mathbb{R}$ is \textit{supported} on a Borel set~$S$ if $\nu(\mathbb{R}^n\setminus S)=0$. 
Finally, it will also be convenient to use the symbol~$\mathrm K$ to refer to either~$\mathrm H$ or~$\mathrm P$, which stands for Hausdorff and packing properties, respectively.

\section{Preliminaries} \label{HPM}

\subsection{Hausdorff and packing measures}

Let us recall the definitions of Hausdorff and packing measures on~$\mathbb{R}$.

\DEFI Let $A\subset\mathbb{R}$. By a $\delta$-covering of~$A$ we mean any countable collection~$\{E_k\}$ of subsets of~$\mathbb{R}$ such that~$A\subset\cup_{k\ge 1}E_k$ and $\diam(E_k):=\sup_{x,y\in E_k}\vert x-y\vert\le\delta$. For each $\alpha\in[0,1]$, the~$\alpha$-\textit{dimensional (exterior)  Hausdorff measure} of~$A$ is defined as
\begin{equation*} h^{\alpha}(S) = \lim_{\delta \downarrow 0}\inf \Big\{ \sum_{ k  =1}^{\infty }\diam(E_k)^{\alpha}\mid \{E_k\}\; \mathrm{is\; a\;} \delta\text{-covering\; of\;} S\Big\} \;.  \label{Qhaus} \end{equation*}
\label{DHaus}
\DEFF

The \textit{Hausdorff dimension} of the set~$S$, here denoted by~$\dim_{\mathrm H}(S)$, is defined as the infimum of all~$\alpha$ such that  $h^\alpha(S)=0$; note that $h^\alpha(S)=\infty$ if $\alpha<\dim_{\mathrm H}(S)$.

 A $\delta$-packing of an arbitrary set $S\subset\mathbb{R}$ is a countable disjoint collection $(\bar B(x_k;r_k))_{k\in\mathbb{N}}$ of closed balls centered at~$x_k\in S$ and radii $r_k\le\delta/2$, so with diameters at most of~$\delta$. Define $P^\alpha_\delta(S)$, $\alpha\in[0,1]$, as
\begin{equation*}
 P^\alpha_\delta(S)=\sup\Big\{\sum_{k=1}^\infty(2r_k)^\alpha\mid (\bar B(x_k;r_k))_k \;\mathrm{is\; a} \;\delta\text{-packing\;of}\;A\Big\}\;,
\end{equation*}
that is,  the supremum is taken over all~$\delta$-packings of~$S$. Then, take the decreasing limit
\begin{equation*}
 P_0^\alpha(S)=\lim_{\delta \downarrow 0}P^\alpha_\delta(S)
\end{equation*} which is  a pre-measure.

\DEFI
The \textit{$\alpha$-packing (exterior) measure} $P^\alpha(S)$ of~$S$ is given by
\begin{equation*}
 P^\alpha(S):=\inf\Big\{\sum_{k=1}^\infty P_0^\alpha(E_k)\mid S\subset\bigcup_{k=1}^\infty E_k \Big\} \;. 
\end{equation*}\label{DPack}
\DEFF
 
The  \textit{packing dimension} of the set~$S$, here denoted by~$\dim_{\mathrm P}(S)$, is defined (in analogy to~$\dim_{\mathrm H}(S)$) as the infimum of all~$\alpha$ such that  $P^\alpha(S)=0$, which coincides with the supremum of all~$\alpha$ so that~$P^\alpha(S)=\infty$. 

It is known~\cite{Mat} that  $\dim_{\mathrm H}(S)\le\dim_{\mathrm P}(S)$, and this 
inequality is in general strict. It is also important to mention that $P^\alpha$ and $h^\alpha$ are Borel (regular) measures; furthermore, $P^0\equiv h^0$, $P^1\equiv h^1$, and they are equivalent, respectively, to the counting measure (which assigns to each set~$S$ the number of elements it has) and the Lebesgue measure. 

\DEFI\label{D4.5L}
Let $\alpha\in[0,1]$. A finite nonnegative Borel measure~$\mu$ on~$\mathbb{R}$ is called:
\begin{enumerate}
\item  \textit{$\alpha\text{-}\mathrm{K}$ continuous}, denoted $\alpha\mathrm{Kc}$, if~$\mu(S)=0$ for every Borel set~$S$ such that 
  $\mathrm{K}^\alpha(S)=0$. 
\item  \textit{$\alpha\text{-}\mathrm{K}$ singular}, denoted $\alpha\mathrm{Ks}$, if it is supported on some Borel set~$S$  with $\mathrm{K}^{\alpha}(S)=0$. 
\item 0-$\mathrm{K}$ {\em dimensional}, denoted $0\mathrm{Kd}$, if it is supported on a Borel set~$S$ with $\dim_{\mathrm K}(S)=0$.
\item 1-$\mathrm{K}$ {\em dimensional}, denoted $1\mathrm{Kd}$, if $\mu(S)=0$ for any Borel set~$S$ with $\dim_{\mathrm K}(S)<1$. 
\end{enumerate}
\DEFF

\OBSI\label{Obs}
\begin{enumerate}
\item $\mu$ is~$0\mathrm{Kd}$ if, and only if, it is~$\alpha$Ks for each~$\alpha\in(0,1]$. Equivalently,~$\mu$ is~$1\mathrm{Kd}$ if, and only if, it is~$\alpha$Kc for each $\alpha\in[0,1)$.
\item It follows from Definition~\ref{D4.5L} that~$\mu$ is~$0\mathrm{Kd}$ if it is pure point, whereas~$\mu$ is~$1\mathrm{Kd}$ if it is absolutely continuous.
\end{enumerate}
\OBSF

\DEFI \label{LFD1}
Let~$\mu$ be a finite nonnegative Borel measure on~$\mathbb{R}$ and $x\in\mathbb{R}$. Set $B(x;\varepsilon)=\{y\in\mathbb{R}\mid \vert x-y\vert<\varepsilon\}$, i.e., the open ball of radius~$\varepsilon>0$ centered at $x$, and
\[
D_{\mu}^{\mathrm{H},\alpha}(x):=\limsup_{\varepsilon\downarrow 0}\frac{\mu(B(x;\varepsilon))}{(2\varepsilon)^\alpha}\;,\qquad D_{\mu}^{\mathrm{P},\alpha}(x):=\liminf_{\varepsilon\downarrow 0}\frac{\mu(B(x;\varepsilon))}{(2\varepsilon)^\alpha}\;.
 \] 
\DEFF

The following density results~\cite{Guar,Rogers} relate the continuity of~$\mu$, with respect to Hausdorff (packing) dimension, to its local scaling behavior as probed by~$D_{\mu}^{\mathrm{K},\alpha}$. 

\begin{theorem}\label{Corl}
Let~$\mu$ be as above and let $\alpha\in[0,1]$. Let 
\begin{align}
\nonumber&K_{\mathrm{\alpha Kc}}:=\{x\in\mathbb{R}\mid D^{\mathrm{K},\alpha}_\mu(x)<\infty\},\qquad
K_{\mathrm{\alpha Ks}}:=\{x\in\mathbb{R}\mid D^{\mathrm{K},\alpha}_\mu(x)=\infty\}.
\end{align}
Then, these are Borel sets, $\mu_{\mathrm{\alpha Kc}}(\cdot):= \mu(K_{\mathrm{\alpha Kc}}\cap\cdot)$ is $\alpha\mathrm{Kc}$,~$\mu_{\mathrm{\alpha Ks}}(\cdot):= \mu(K_{\mathrm{\alpha Ks}}\cap\cdot)$ is $\alpha\mathrm{Ks}$,~$\mu_{0\mathrm{Kd}}(\cdot):= \mu((\bigcap_{k\ge 1}K_{\mathrm{(1/k)Ks}})\cap\cdot)$ is $0\mathrm{Kd}$, and $\mu_{1\mathrm{Kd}}(\cdot):= \mu((\bigcap_{k\ge 1}K_{\mathrm{(1-1/k)Kc}})\cap\cdot)$ is $1\mathrm{Kd}$. 
\end{theorem}
\begin{proof}
See Section 4 in~\cite{Last} for the Hausdorff case; the packing case follows analogously.
\end{proof}

By following~\cite{Guar,You}, we recall the  upper and lower dimensions of a finite Borel measure~$\mu$.

\DEFI\label{defDimMu}
Let~$\mu$ be as above, and let $I \subset\mathbb{R}$ be a Borel set. The $\mathrm K$ {\it upper dimension} of~$\mu$ restricted to~$I$, denoted by $\dim_{\mathrm K}^+(\mu_{;I})$, is defined as 
\[
\dim_{\mathrm K}^+(\mu_{;I}):=\inf \{ \dim_{\mathrm K}(S)\mid \mu(I\setminus S)=0, \,S\; \mathrm{a\; Borel\; subset\; of}\; I \},
\]
and the $\mathrm K$ {\it lower dimension} of~$\mu$ restricted to~$I$, denoted by $\dim_{\mathrm K}^-(\mu_{;I})$, as
\[
\dim_{\mathrm K}^-(\mu_{;I}):=\sup\{\alpha\mid\mu(S)=0\;\;\mathrm{if}\;\dim_{\mathrm K}(S)<\alpha,\;S\; \mathrm{a\; Borel\; subset\; of}\; I\}.
\]
When $I=\mathbb{R}$, we simply denote $\dim_{\mathrm K}^{\pm}(\mu_{;I})$ by $\dim_{\mathrm K}^{\pm}(\mu)$. 
\DEFF

\begin{proposition}
\label{CPRo1}
Let $\mu$ be as above, let~$I$ be a Borel subset of $\mathbb{R}$, and let~$\alpha\in(0,1)$. Then,
\begin{enumerate}
\item $\alpha\le\dim_{\mathrm K}^-(\mu_{;I})$ if, and only if, for each~$\var\in(0,\alpha]$,~$\mu_{;I}$ is~$\mathrm{(\alpha-\var)Kc}$;
\item $\dim_{\mathrm K}^+(\mu_{;I})\le\alpha$ if, and only if, for each~$\var\in(0,1-\alpha]$,~$\mu_{;I}$ is $(\alpha+\var)\mathrm{Ks}$.
\end{enumerate}
\end{proposition}
\begin{proof} See Section~1 in~\cite{SCC}.
\end{proof}

\section{$G_\delta$ sets}\label{sectHW}

Let $(X,d)$ be as in the Introduction, let $\emptyset\neq O\subset\mathbb{R}$ be an open set, 
and let
\[\mathcal{M}_+(O):=\Big\{\mu\in \mathcal{M}(O)\mid0\le\mu\le 1
\Big\},\]
that is, the set of positive measures on $O$ with total mass less than or equal to one. 
We endow such set with the weak topology, i.e., the topology of the weak convergence of measures ($(\mu_n)$ converges weakly to $\mu$ if for each $f\in C_b(O)$, $\int f(x)d\mu_n(x)\rightarrow \int f(x)d\mu(x)$; here, $C_b(O)$ denotes the set of bounded continuous functions defined on $O$). Recall that such topology is metrizable (since $O$ is a Polish space): 
take, for instance, the L\'evy-Prohorov metric, which will be denoted by $\rho$ (see Appendix 2 in~\cite{Daley} for details).

Let also, for each $T\in X$ and each $0\neq\psi\in\mathcal{H}$, $\zeta_\psi:X\rightarrow\mathcal{M}_+(O)$ be defined by the law $\zeta_\psi(T):=\mu_{\psi;O}^{T}$, where $\mu_{\psi;O}^{T}(\cdot):=\mu_{\psi}^{T}(O\cap\cdot)$. 
It follows from the functional calculus for self-adjoint operators that $\zeta_\psi$ is a continuous function: if $\lim_{m\to\infty}d(T_m,T)=0$, then $\lim_{m\to\infty}\rho\left(\mu_{\psi;O}^{T_m},\mu_{\psi;O}^{T}\right)=0$.

\begin{lemma}
  \label{asympt1}
  Let $\emptyset\neq O\subset\mathbb{R}$ be an open set and let, for each~$t>0$, $V_{t}(\cdot,\cdot):\mathcal{M}_+(O)\times O\rightarrow[0,1]$ be defined by the law $V_{t}(\mu,x):=\int f_{t,x}(y)d\mu(y)$, where  $f_{t,x}:O\rightarrow[0,1]$ is given by 
 $$
  f_{t,x}(y):= \left\{ \begin{array}{lcc}                  
            1 & ,if  & |x-y| \leq 1/t, \\
            \\ -t|x-y|+2&, if & 1/t\le |x-y|\le 2/t,\\
            \\ 0 &   ,if  & |x-y| \geq 2/t.
          \end{array}
\right.$$
Let also, for each $0\neq\psi\in\mathcal{H}$, $U_{t,\psi}(\cdot,\cdot):X\times O\rightarrow[0,1]$ be defined by the law 
\[
U_{t,\psi}(T,x):=(\psi,f_{t,x}(T)\psi)=\int f_{t,x}(y)d\mu_{\psi;O}^{T}(y).
\]

Then, $U_{t,\psi}(T,x)=V_t(\zeta_\psi(T),x)$ and
\[(D^{\mathrm{K},\alpha}\mu_{\psi;O}^{T})(x)=\limK_{t\to\infty}t^\alpha U_{t,\psi}(T,x).\]
Furthermore, for each $t>0$, the function $V_{t}:\mathcal{M}_+(O)\times O\rightarrow [0,1]$ is jointly continuous. 
\end{lemma}
\begin{proof}
It follows from the the Spectral Theorem that, for each $x\in O$, each $t>0$ and each $0\neq\psi\in\mathcal{H}$, 
\begin{eqnarray*}
\mu_{\psi;O}^{T}(B_{1/t}(x))\le U_{t,\psi}(T,x)=\int f_{t,x}(y)d\mu_{\psi;O}^{T}(y)\le \mu_{\psi;O}^{T}(B_{2/t}(x)). 
\end{eqnarray*} 
Then, 
one has $t^\alpha\mu_{\psi;O}^{T}(B(x,1/t))\le t^\alpha U_{t,\psi}(T,x)\le t^\alpha\mu_{\psi;O}^{T}(B(x,2/t))$, which proves the first assertion. 

Note that, for each $x\in O$ and each $t>0$, $f_{t,x}: O\rightarrow \mathbb{R}$ is a continuous function such that, for each $y\in O$, $ \chi_{_{B(x,1/t)}}(y) \leq f_{t,x}(y)\leq \chi_{_{B(x,2/t)}}(y)$. Given that each $f_{t,x}(y)$ depends only on $|x-y|$, it is straightforward to show that for each $t>0$, $f_{t,x_l}$ converges uniformly to $f_{t,x}$ on $O$ when $x_l\rightarrow x$.

We combine this remark with Theorems 2.13 and 2.15 in \cite{Habil} in order to prove that $V_{t}(\mu,x)$ is jointly continuous. Let $(\mu_m)$ and $(x_l)$ be sequences in $\mathcal{M}_+(O)$ and $O$, respectively, such that $\rho(\mu_m,\mu)\rightarrow 0$ and  $x_l\to x$. Firstly, we show that 
$$\lim_{m\to \infty}\lim_{l\to\infty} U_{t,\psi}(\mu_m,x_l)=\lim_{m\to \infty}\lim_{l\to\infty}\int f_{t,\psi}(y)d\mu_m(y)=V_{t}(\mu,x).$$

Since, for each $y\in \mathbb{R}$, $|f_{t,x_l}(y)|\leq 1$, it follows from dominated convergence that, for each $m\in\mathbb{N}$, $\lim_{l\to\infty}\int f_{t,x_l}(y)d\mu_m(y)= \int f_{t,x}(y)d\mu_m(y)$. 
Now, since $f_{t,x}$ is continuous and convergence in the metric~$\rho$ implies weak convergence of measures, one has
\begin{equation*}
\lim_{m\to \infty}\lim_{l\to\infty}\int f_{t,x_l}(y)d\mu_m(y)= \lim_{m\to \infty}\int f_{t,x}(y)d\mu_m(y)=V_{t}(\mu,x).
\end{equation*}

The next step consists in showing that, for each $l\in \mathbb{N}$, the function $\varphi_l:\mathbb{N}\rightarrow \mathbb{R}$, defined by the law $\varphi_l(m):= V_{t}(\mu_m,x_l)$,  converges  uniformly to $\varphi(m):=\lim_{l\to \infty}V_{t}(\mu_m,x_l)=\int f_{t,x}(y)d\mu_m(y)$. Let $\delta>0$. Since, for each $t>0$, $f_{t,x_l}(y)$ converges uniformly to $f_{t,x}(y)$, there exists $N\in\mathbb{N}$ such that, for each $l\ge N$ and each $y\in \mathbb R$, $\left| f_{t,x_l}(y)- f_{t,x}(y)\right|<\delta$. Then, one has, for each $l\geq N$ and each $m\in \mathbb{N}$, 
\begin{eqnarray*}
|\varphi_l(m)-\varphi(m)|&=&\left| \int f_{t,x_l}(y)d\mu_m(y)- \int f_{t,x}(y)d\mu_m(y) \right|
\\ &\leq&  \int \left| f_{t,x_l}(y)- f_{t,x}(y)\right| d\mu_m(y)< \delta.
\end{eqnarray*}

It follows from Theorem  2.15 in \cite{Habil} that $\lim_{l,m \to \infty} V_{t}(\mu_m,x_l)= V_{t}(\mu,x)$. Given that \linebreak $\lim_{l\to\infty} V_{t}(\mu_m,x_l)=\int f_{t,x}(y)d\mu_m(y)$ and that $\lim_{m\to\infty} V_{t}(\mu_m,x_l)=\int f_{t,x_l}(y)d\mu(y)$ exist for each $m\in\mathbb{N}$ and each $l\in\mathbb{N}$,  respectively, Theorem 2.13 in \cite{Habil} implies that 
$$\lim_{m\to \infty}\lim_{l\to\infty} V_{t}(\mu_m,x_l)=\lim_{l\to \infty}\lim_{m\to\infty} V_{t}(\mu_m,x_l)=\lim_{l,m \to \infty} V_{t}(\mu_m,x_l)=V_{t}(\mu,x).$$ 

Hence, if $(\mu_l,x_l)$ is some sequence in $\mathcal{M}_+(O)\times O$ (endowed with the product topology) such that $(\mu_l,x_l)\to (\mu,x)\in\mathcal{M}_+(O)\times O$, then $\lim_{l \to \infty} V_{t}(\mu_l,x_l)=V_{t}(\mu,x)$, showing that $V_{t}(\,\cdot\,,\,\cdot\,)$ is jointly continuous at $(\mu,x)$.
\end{proof}

\

Before we present our main result, some preparation is required. 
Let, for each $\alpha\in(0,1)$, $\beta^{\mathrm{H},\alpha}_{\mu}:E\times\mathbb{N}\rightarrow[0,+\infty)$ be defined by the law $\beta^{\mathrm{H},\alpha}_{\mu}(x,s):=\sup_{t\ge s}t^\alpha V_{t}(\mu,x)$, where for each $t>0$, $V_{t}(\cdot,\cdot):\mathcal{M}_+(E)\times E$ is defined as in the statement of Lemma~\ref{asympt1}. 

  \OBSI\label{Rasympt1}
  The proof that, for each $t>0$, the mapping $V_{t}(\cdot,\cdot):\mathcal{M}_+(E)\times E$, $V_t(\mu,x)=\int f_{t,x}(y) d\mu(y)$, is jointly continuous if $(E,d)$ is a Polish metric space is identical to the proof of Lemma~\ref{asympt1};  in the definition of $f_{t,x}$, just replace the euclidean metric in~$\mathbb{R}$ by~$d$. 
\OBSF

\begin{lemma}\label{LIMP}
  Let $E$ be a Polish metric space and let $\alpha\in(0,1)$.  Then, for each $\delta>0$ and each $r,s\in\mathbb{N}$, 
  \[\mathcal{M}_{r,s}(\delta):=\{\mu\in\mathcal{M}_+(E)\mid\mu(Z_\mu(r,s))\ge \delta\}\] is a closed subset of $\mathcal{M}(E)$, where $Z_\mu(r,s):=\{x\in E\mid \beta^{\mathrm{H},\alpha}_{\mu}(x,s)\le r\}$. 
\end{lemma}
\begin{proof}
  {\it{Claim 1.}}~For each $r,s\in\mathbb{N}$ and each $\mu\in\mathcal{M}_+(E)$, $Z_\mu(r,s)$ is a closed subset of $E$.
  
Let~$\{w_i\}$ be a sequence in $Z_\mu(r,s)$ such that~$\lim w_i=w$. Since, for each~$t>0$, $f_{t,w_i}\rightarrow f_{t,w}$ pointwise, it follows from Remark~\ref{Rasympt1} that the mapping $x\mapsto\beta^{\mathrm{H},\alpha}_{\mu}(x,s)$ is lower semi-continuous. Hence, $\beta^{\mathrm{H},\alpha}_{\mu}(w,s)\le r$, which means that $w\in Z_{\mu}(r,s)$.

\

\textit{Claim 2.} For each $s\in\mathbb{N}$, $W_{r,s}=\{(\nu,x)\in \mathcal{M}(E)\times E\mid \beta^{H,\alpha}_{\nu}(x,s)>r\}$ is open.

This is a consequence of the fact that, by Remark~\ref{Rasympt1}, the mapping $\mathcal{M}_+(E)\times E\ni(\nu,x)\longmapsto \beta^{H,\alpha}_{\nu}(x,s)$
is lower semi-continuous. 

\

Now, we show that $\mathcal{M}_{r,s}(\delta)$ is closed. Let $\mu_m$ be a sequence in $\mathcal{M}_{r,s}(\delta)$ such that $\mu_m\to\mu$. Suppose, by absurd, that $\mu\notin \mathcal{M}_{r,s}(\delta)$; we will find that $\mu_m\notin \mathcal{M}_{r,s}(\delta)$ for~$m$ sufficiently large, a contradiction.  

If $\mu\notin\mathcal{M}_{r,s}(\delta)$, then $\mu(A)>\mu(E)-\delta$, where $A:=E\setminus Z_{\mu}(r,s)$. Hence $\{\mu\}\times A\subset W_{r,}$.
Since $\mu$ is a tight measure on $E$ ($\mu$ is a Borel measure and the space $X$ is Polish; see Proposition~A.2.2.V in~\cite{Daley}), there exists a compact $C\subset A$ such that $\mu(C)>\mu(E)-\delta$ (note that, by Claim 1, $A$ is open).

Now, we construct a suitable subset of $W_{r,s}$ that contains a neighborhood of $\{\mu\}\times C$. Let, for each $x\in C$, $V_x\subset W_{r,s}$ be an open neighborhood of~$(\mu,x)$ (such open set exists, by Claim 2); that is, $V_x:=B((\mu,x);\varepsilon)=\{(\nu,y)\in\mathcal{M}_+(E)\times E\mid\max\{\rho(\nu,\mu),d(x,y)\}<\varepsilon\}$, for some suitable $\varepsilon>0$. Then, $\{V_x\}_{x\in C}$ is an open cover of $\{\mu\}\times C$, and since $\{\mu\}\times C$ is a compact subset of $\mathcal{M}_+(E)\times E$, it follows that one can extract from $\{V_x\}_{x\in C}$ a finite subcover, $\{V_{x_i}\}_{i=1}^n$.

We affirm that there exists an $\ell\in\mathbb{N}$ (which depends on $C$) such that $\{\mu_n\}_{n\ge\ell}\subset\bigcap_{i}(\pi_1(V_{x_i}))$. Namely, for each $i$, there exists an $\ell_i$ such that $\{\mu_n\}_{n\ge\ell_i}\subset\pi_1(V_{x_i})$; set $\ell:=\max\{\ell_i\mid i\in\{1,\ldots,n\}\}$, and note that for each $i$, $\{\mu_n\}_{n\ge\ell}\subset\pi_1(V_{x_i})$. Set also $\mathcal{I}:=\bigcap_{i}(\pi_1(V_{x_i}))$ and $\mathcal{O}:=\bigcup_{i}(\pi_2(V_{x_i}))$.

Since for each $i$, $V_{x_i}=\pi_1(V_{x_i})\times\pi_2(V_{x_i})$, and given that
\[\{\mu_n\}_{n\ge\ell}\times\mathcal{O}\subset\mathcal{I}\times\mathcal{O}\subset\bigcup_{i}(\pi_1(V_{x_i})\times\pi_2(V_{x_i}))=\bigcup_{i}V_{x_i}\subset W_{r,s},\]
it follows that, for each $n\ge\ell$ and each $y\in\mathcal{O}$, $\beta^{\mathrm{H},\alpha}_{\mu_m}(y,s)>r$. Moreover, $\mathcal{O}$ is an open set that contains $C$.

On the other hand, weak convergence implies that 
\[
\displaystyle\limsup_{m\to \infty}\mu_m(E\setminus\mathcal{O})\leq\mu(E\setminus\mathcal O)\leq  \mu(E\setminus C)<\delta,
\] from which follows that there exists an $\ell_1\ge\ell$ such that, for  $m\ge\ell_1$, $\mu_m(E\setminus \mathcal{O})<\delta$.

Combining the last results, one concludes that, for  $m\ge\ell_1$, $\mu_m(E\setminus \mathcal{O})<\delta$, and for each $x\in \mathcal{O}$, $\beta^{\mathrm{H},\alpha}_{\mu_m}(x,s)>r$, so
\[\mu_m(Z_{\mu_m}(r,s))\le\mu_m(E\setminus\mathcal{O})<\delta;\]
this contradicts the fact that, for each $m\in\mathbb{N}$, $\mu_m\in \mathcal{M}_{r,s}(\delta)$. Hence, $\mu \in  \mathcal{M}_{r,s}(\delta)$, and $\mathcal{M}_{r,s}(\delta)$ is a closed subset of $\mathcal{M}_+(E)$. 
\end{proof}  

\

Define, for  $\alpha\in(0,1)$ and  $s\in\mathbb{N}$, $\gamma^{\mathrm{H(P)},\alpha}_{\psi,T}(x,s):=\sup(\inf)_{t\ge s}t^\alpha U_{t,\psi}(T,x)$. Then, by Lemma~\ref{asympt1}, one has, for each $x\in O$, $\lim_{s\to\infty}\gamma^{\mathrm{K},\alpha}_{\psi,T}(x,s)=(D^{\mathrm{K},\alpha}\mu_{\psi;O}^{T})(x)$. 
By definition, for each $x\in O$, $\mathbb{N}\ni s\mapsto\gamma^{\mathrm{H(P)},\alpha}_{\psi,T}(x,s)\in[0,+\infty)$ is a nonincreasing (nondecreasing) mapping.

\begin{theorem} \label{P2}
  Let~$\emptyset\neq F\subset\mathbb{R}$ be a closed subset, let 
  $0\neq\psi\in\mathcal{H}$, and  $\mu_{\psi;F}^T(\cdot):=\mu_{\psi}^T(F\cap\cdot)$.
  Then, each of the sets~$C^{\psi;F}_{0\mathrm{Hd}}:=\{T\in X\mid\dim_{\mathrm{H}}^+(\mu_{\psi;F}^{T})=0\}$ and 
  $C^{\psi;F}_{1\mathrm{Pd}}:=\{T\in X\mid\dim_{\mathrm{P}}^-(\mu_{\psi;F}^{T})=1\}$ is a $G_\delta$ set in~$X$. 
\end{theorem}
\begin{proof}
  Since the arguments in both proofs are analogous, we just prove the statement for~$C^{\psi;F}_{0\mathrm{Hd}}$. Note that for each closed set $F$, there exists a countable family of open sets, $\{A_i\}$, such that $F=\bigcap_{i\ge 1}A_i$ (each closed set $F$ is a $G_\delta$ set); thus, one just has to prove the result for $C^{\psi;O}_{0\mathrm{Hd}}$, where $\emptyset\neq O\subset\mathbb{R}$ is an open set.

  If, for each $T\in X$, $\mu_{\psi;O}^{T}(\mathbb{R})=0$, then $C^{\psi;O}_{0\mathrm{Hd}}=\emptyset$ is a $G_\delta$ subset of $X$. Thus, suppose that $\{T\in X\mid\mu_{\psi;O}^{T}(\mathbb{R})>0\}\neq\emptyset$.

  Set, for each~$\alpha\in(0,1)$, 
$C^{\psi;O}_{\mathrm{\alpha Hc}}:=
  \bigcup_{p\ge 1}C^{\psi;O}_{\mathrm{\alpha Hc}}(p)$, where $C^{\psi;O}_{\mathrm{\alpha Hc}}(p)=\{T\in X\mid\mu_{\psi;O}^{T}(\{x\in \mathbb{R}\mid(D^{\mathrm{H},\alpha}\mu_{\psi;O}^{T})(x)< p\})>0\}$. Now, by 
  Theorem~\ref{Corl} and Proposition~\ref{CPRo1},
\begin{equation}\label{igcon}
    C^{\psi;O}_{0\mathrm{Hd}}=
  \bigcap_{k>1} (C^{\psi;O}_{\mathrm{(1/k)Hc}})^c=\bigcap_{k>1}\bigcap_{p\ge1} (C^{\psi;O}_{\mathrm{(1/k)Hc}}(p))^c.
\end{equation}

{\textit{Claim 1.}} For each $\alpha\in(0,1)$ and each $p\in\mathbb{N}$,
\[
(C^{\psi;O}_{\mathrm{\alpha Hc}}(p))^c=\bigcap_{s\in\mathbb{N}}\{T\in X\mid \mu_{\psi;O}^{T}\textrm{-}\essinf\gamma_{\psi,T}^{\mathrm{H},\alpha}(x,s)\ge p\}.
\]

Let $T\in (C^{\psi;O}_{\mathrm{\alpha Hc}}(p))^c$.  Since, for each $x\in\mathbb{R}$,  $\mathbb{N}\ni s\mapsto\gamma_{\psi,T}^{\mathrm{H},\alpha}(x,s)\in[0,+\infty)$ is a nonincreasing function, it follows that, for each $s\in\mathbb{N}$, $\mu_{\psi;O}^{T}\textrm{-}\essinf\gamma_{\psi,T}^{\mathrm{H},\alpha}(x,s)\ge p$.

Now, let $T\in\bigcap_{s\in\mathbb{N}}\{U\in X\mid \mu_{\psi;O}^{U}\textrm{-}\essinf\gamma_{\psi,U}^{\mathrm{H},\alpha}(x,s)\ge p\}$. Then, for each $s\in\mathbb{N}$, there exits a Borel set $A_s\subset\mathbb{R}$, with $\mu_{\psi;O}^{T}(A_s)=1$, such that for each $x\in A_s$, $\gamma_{\psi,T}^{\mathrm{H},\alpha}(x,s)\ge p$. Let $A:=\bigcap_{s\ge 1}A_s$; then, for each $x\in A$, one has $(D^{\mathrm{H},\alpha}\mu_{\psi;O}^{T})(x)=\lim_{s\to\infty}\gamma_{\psi,U}^{\mathrm{H},\alpha}(x,s)\ge p$; given that $\mu_{\psi;O}^{T}(A)=1$, we are done.

\

Let, for each~$\alpha\in(0,1)$ and each $p,q,s,l\in\mathbb{N}$,
\[C^{\psi;O}_{\mathrm{\alpha Hc}}(p-1/q,s,l):=\{T\in X\mid\mu_{\psi;O}^{T}(A_{\psi;O}^{T}(p-1/q,s))\ge 1/l\},\]
where $A_{\psi;O}^{T}(p-1/q,s):=\{x\in O\mid \gamma^{\mathrm{H},\alpha}_{\psi,T}(x,s)\le p-1/q\}$. 
Thus, according to Claim 1 and~\eqref{igcon},
\[C^{\psi;O}_{0\mathrm{Hd}}=\bigcap_{k>1}\bigcap_{p\ge1}\bigcap_{q\ge1}\bigcap_{s\ge 1}\bigcap_{l\ge 1}(C^{\psi;O}_{\mathrm{(1/k)Hc}}(p-1/q,s,l))^c,
\]
and one just needs to show that, for each $\alpha\in(0,1)$ and each $r,s,l\in\mathbb{N}$, $C^{\psi;O}_{\mathrm{\alpha Hc}}(r,s,l)$ is closed in~$X$.

\

\textit{Claim 2.} For each $\delta>0$ and each $r,s\in\mathbb{N}$, $\{\mu\in\mathcal{M}_+(O)\mid \mu(Z_\mu(r,s))\ge\delta\}$ is a closed subset of $\mathcal{M}_+(O)$, where $Z_\mu(r,s)=\{x\in O\mid\beta^{\mathrm{H},\alpha}_{\mu}(x,s)\le r\}$.

Here, we use the fact that $O$ can be isometrically embedded in $\overline{O}$, which is a Polish metric space. Thus, any $\mu\in\mathcal{M}_+(O)$ can be identified with the measure $\tilde{\mu}\in \mathcal{M}_+(\overline{O})$ defined by $\tilde{\mu}(A)=\mu(A\cap O)$ for each $A\in\mathcal{B}(\overline{O})$, and $\mathcal{M}_+(O)$ can be identified with a subset of $\mathcal{M}_+(\overline{O})$, namely, the set $\{\tilde{\mu}\in \mathcal{M}_+(\overline{O})\mid\tilde{\mu}(\overline{O})=\tilde{\mu}(O)\}$.  Then, the induced topology  in $\mathcal{M}_+(O)$ by the Polish space $\mathcal{M}_+(\overline{O})$ coincides with the weak topology in $\mathcal{M}_+(O)$ (see Section~6 in~\cite{Oxtoby} for details).

Moreover, for each $\mu\in\mathcal{M}_+(O)$ and each $r,s\in\mathbb{N}$, $Z_\mu(r,s)=\{x\in\overline{O}\mid\beta^{\mathrm{H},\alpha}_{\tilde{\mu}}(x,s)\le r\}\cap O$, so for each $\delta>0$,
\[\{\mu\in\mathcal{M}_+(O)\mid \mu(Z_\mu(r,s))\ge\delta\}=\{\tilde{\mu}\in\mathcal{M}_+(\overline{O})\mid \tilde{\mu}(Z_{\tilde{\mu}}(r,s))=\mu(Z_\mu(r,s))\ge\delta\}\cap\mathcal{M}_+(O).\] 

The result is now a consequence of Lemma~\ref{LIMP}.

\

Recall that by the functional calculus, for each $0\neq\psi\in\mathcal{H}$,  the mapping $\zeta_\psi:X\rightarrow\mathcal{M}_+(O)$, $\zeta_\psi(T)=\mu_{\psi;O}^{T}$ is continuous (since convergence in~$X$ implies strong resolvent convergence), and note that 
for each $(x,s)\in O\times\mathbb{N}$, $\gamma^{\mathrm{H},\alpha}_{\psi,T}(x,s)=\beta^{\mathrm{H},\alpha}_{\zeta_\psi(T)}(x,s)$. 

Thus, it follows that for each $l,r,s\in\mathbb{N}$, $C^{\psi;O}_{\mathrm{\alpha Hc}}(r,s,l)=(\gamma_\psi)^{-1}(\mathcal{M}_{r,s}(1/l))$, and therefore, by Claim 2, $C^{\psi;O}_{\mathrm{\alpha Hc}}(r,s,l)$ is a closed subset of $X$.
\end{proof}



\section{Proof of Theorems~\ref{thmFDW} and~\ref{LPO}}
\label{sectHW1}
\begin{proof}(Theorem~\ref{thmFDW}) The result is a direct consequence of the hypotheses, Theorem~\ref{P2}, and the fact that the intersection of a countable family of generic sets is still a generic set.
\end{proof}

\

In order to prove Theorem~\ref{LPO}, we need the following result.

\begin{theorem}[Theorems~1.1 in~\cite{DG} and 1.3 in~\cite{DG1}]
  Suppose that~$\Omega$ is a Cantor group and that $\tau:\Omega\rightarrow\Omega$ is a minimal translation. Then, there exist dense sets of $g\in \CC(\Omega,\mathbb{R})$ such that, for each $\kappa\in\Omega$,
\begin{enumerate}
\item the spectrum of $H_{g,\tau}^\kappa$ is purely absolutely continuous;
\item the spectrum of $H_{g,\tau}^\kappa$ is zero-Hausdorff dimensional.
\end{enumerate}  
\label{TDG}
\end{theorem}

\begin{proof}(Theorem~\ref{LPO})
Fix $\kappa\in\Omega$ and let $\tau:\Omega\rightarrow\Omega$ be a minimal translation of the Cantor group~$\Omega$. 

It follows from Theorem~\ref{TDG} that each of the sets  $C_{\mathrm{1Pd}}^\kappa\supset C_{\mathrm{ac}}^\kappa:=\{T\in X_\kappa\mid\sigma(T)$ is purely absolutely continuous$\}$ and $C_{\mathrm{0Hd}}^\kappa$ is dense in $X_\kappa$. 

The result is now a consequence of Theorem~\ref{thmFDW} and Remark~\ref{RIMP}. 
%
\end{proof}

 \subsection*{Acknowledgments} SLC thanks the partial support by FAPEMIG (Universal Project CEX-APQ-00554-13). CRdO  thanks the partial support by CNPq  (under contract 303503/2018-1).

\end{document}